\documentclass[10pt,a4paper,twoside]{amsart}

\usepackage[mathcal]{euscript}
\usepackage{amssymb, amsfonts, xypic, amsmath}
\usepackage{amscd}
\usepackage[all]{xy}
\usepackage{dutchcal}
\usepackage{amsthm}
\usepackage[margin=1in]{geometry}
\usepackage{xcolor}
\usepackage{tikz-cd}

\newtheorem{theorem}{Theorem}

\newtheorem{corollary}[theorem]{Corollary}

\newtheorem{proposition}[theorem]{Proposition}

\theoremstyle{definition}

\newtheorem{remark}[theorem]{Remark}
\newtheorem*{theorem*}{Theorem}

\numberwithin{equation}{section} \numberwithin{figure}{section}

\numberwithin{equation}{section}

\newcommand{\E}{\E_{\infty}}

\usepackage{graphicx}
\usepackage{epstopdf}

\def\YEAR{\year}\newcount\VOL\VOL=\YEAR\advance\VOL by-1995
\def\firstpage{1}\def\lastpage{1000}
\def\received{}\def\revised{}
\def\communicated{}

\makeatletter
\def\magnification{\afterassignment\m@g\count@}
\def\m@g{\mag=\count@\hsize6.5truein\vsize8.9truein\dimen\footins8truein}
\makeatother

\font\eightrm=cmr8
\font\caps=cmcsc10                    % Theorem, Lemma etc
\font\Caps=cmcsc10 scaled \magstep1   % Title

\makeatletter
\renewcommand{\@evenhead}{%
    \ifnum\thepage>\lastpage\rlap{\thepage}\hfill%
    \else\rlap{\thepage}\slshape\leftmark\hfill{\caps\SAuthor}\hfill\fi}%
\renewcommand{\@oddhead}{%
    \ifnum\thepage=\firstpage{\DocMath\hfill\llap{\thepage}}%
    \else{\slshape\rightmark}\hfill{\caps\STitle}\hfill\llap{\thepage}\fi}%
\makeatother

\def\TSkip{\bigskip}
\newbox\TheTitle{\obeylines\gdef\GetTitle #1
\ShortTitle  #2
\SubTitle    #3
\Author      #4
\ShortAuthor #5
\EndTitle
{\setbox\TheTitle=\vbox{\baselineskip=20pt\let\par=\cr\obeylines%
\halign{\centerline{\Caps##}\cr\noalign{\medskip}\cr#1\cr}}%
	\copy\TheTitle\TSkip\TSkip%
\def\next{#2}\ifx\next\empty\gdef\STitle{#1}\else\gdef\STitle{#2}\fi%
\def\next{#3}\ifx\next\empty%
    \else\setbox\TheTitle=\vbox{\baselineskip=20pt\let\par=\cr\obeylines%
    \halign{\centerline{\caps##} #3\cr}}\copy\TheTitle\TSkip\TSkip\fi%
\centerline{\caps #4}\TSkip\TSkip%
\def\next{#5}\ifx\next\empty\gdef\SAuthor{#4}\else\gdef\SAuthor{#5}\fi%
\ifx\received\empty\relax
    \else\centerline{\eightrm Received: \received}\fi%
\ifx\revised\empty\TSkip%
    \else\centerline{\eightrm Revised: \revised}\TSkip\fi%
\ifx\communicated\empty\relax
    \else\centerline{\eightrm Communicated by \communicated}\fi\TSkip\TSkip%
\catcode'015=5}}\def\Title{\obeylines\GetTitle}
\def\Abstract{\begingroup\narrower
    \parskip=\medskipamount\parindent=0pt{\caps Abstract. }}
\def\EndAbstract{\par\endgroup\TSkip}

\long\def\MSC#1\EndMSC{\def\arg{#1}\ifx\arg\empty\relax\else
     {\par\narrower\noindent%
     2010 Mathematics Subject Classification: #1\par}\fi}

\long\def\KEY#1\EndKEY{\def\arg{#1}\ifx\arg\empty\relax\else
	{\par\narrower\noindent Keywords and Phrases: #1\par}\fi\TSkip}

\newbox\TheAdd\def\Addresses{\vfill\copy\TheAdd\vfill
    \ifodd\number\lastpage\vfill\eject\phantom{.}\vfill\eject\fi}
{\obeylines\gdef\GetAddress #1
\Address #2 
\Address #3
\Address #4
\EndAddress
{\def\xs{4.3truecm}\parindent=0pt
\setbox0=\vtop{{\obeylines\hsize=\xs#1\par}}\def\next{#2}
\ifx\next\empty % 1 address
     \setbox\TheAdd=\hbox to\hsize{\hfill\copy0\hfill}
\else\setbox1=\vtop{{\obeylines\hsize=\xs#2\par}}\def\next{#3}
\ifx\next\empty % 2 addresses
     \setbox\TheAdd=\hbox to\hsize{\hfill\copy0\hfill\copy1\hfill}
\else\setbox2=\vtop{{\obeylines\hsize=\xs#3\par}}\def\next{#4}
\ifx\next\empty\ % 3 addresses
     \setbox\TheAdd=\vtop{\hbox to\hsize{\hfill\copy0\hfill\copy1\hfill}
                \vskip20pt\hbox to\hsize{\hfill\copy2\hfill}}
\else\setbox3=\vtop{{\obeylines\hsize=\xs#4\par}}
     \setbox\TheAdd=\vtop{\hbox to\hsize{\hfill\copy0\hfill\copy1\hfill}
	        \vskip20pt\hbox to\hsize{\hfill\copy2\hfill\copy3\hfill}}
\fi\fi\fi\catcode'015=5}}\gdef\Address{\obeylines\GetAddress}

\def\LOCAL{\jobname.files}

\begin{document}
%%%%% ------------- fill in your data below this line  -------------------
%%%%%    The following lines \Title ... \EndAddress must ALL be present
%%%%%    and in the given order.
\Title Adams' cobar construction revisited
%%%%%    Put here the title. Line breaks will be recognized. 
\ShortTitle Adams' cobar construction
%%%%%    Running title for odd numbered pages, ONE line, please. 
%%%%%    If none is given, \Title will be used instead.          
\SubTitle   
%%%%%    A possible subtitle goes here.
\Author Manuel Rivera
%%%%%    Put here name(s) of authors. Line breaks will be recognized.  
\ShortAuthor 
%%%%%%   Running title for even numbered pages, ONE line, please. 
%%%%%%   If none is given, \Author will be used instead.          
\EndTitle
\Abstract We give a short and streamlined proof of the following statement recently proven by the author and M. Zeinalian: the cobar construction of the dg coassociative coalgebra of normalized singular chains on a path-connected pointed space is naturally quasi-isomorphic as a dg associative algebra to the singular chains on the based loop space. This extends a classical theorem of F. Adams originally proven for simply connected spaces. Our proof is based on relating the cobar functor to the left adjoint of the homotopy coherent nerve functor.
%%%%%    Put here the abstract of your manuscript.
\EndAbstract
\MSC 57T30, 55P35, 57T25
%%%%%    2010 Mathematics Subject Classification: 
\EndMSC
\KEY cobar construction, based loop space
%%%%%    Keywords and Phrases:     
\EndKEY
%%%%%    All 4 \Address lines below must be present. To center the last
%%%%%    entry, no empty lines must be between the following \Address
%%%%%    and \EndAddress lines.
\Address 
%%%%%    Address of first Author here
\Address Purdue University, 150 N University St., West Lafayette, IN, USA, 47906
%%%%%    Address of second Author here etc.
\Address
\Address
\EndAddress
%%
%%       Make sure the last tex command in your manuscript
%%       before the first \end{document} is the command  \Addresses
%%

\section{Introduction}
In 1956, F. Adams described a natural algebraic construction, called the \textit{cobar construction}, which passes from a differential graded (dg) coassociative coalgebra to a differential graded (dg) associative algebra \cite{Ad56}. Adams showed that when applied to a suitable chain model of a simply connected space, the cobar construction yields a chain model for its based loop space. Both the cobar construction and its relationship to the based loop space have been of fundamental importance in various contexts such as homological algebra, rational homotopy theory, the theory of operads, symplectic geometry, and representation theory,  to name a few.
\medskip

In 2016, the author and M. Zeinalian proved that Adams' theorem extends to any path-connected space \cite{RiZe16}. This result was then used to explain the sense in which the algebraic structure of the singular chains determines the fundamental group and to show that the dg coalgebra of singular chains detects weak homotopy equivalences between spaces \cite{RiZe18},\cite{RWZ18}. In this article we provide a short, direct, and streamlined proof of this extension of Adams' theorem, which was obtained in \cite{RiZe16} from deeper categorical and combinatorial considerations. In particular, in contrast to \cite{RiZe16}, in this article we do not make use of the theory of cubical sets with connections.
\medskip

Let $(X,b)$ be a pointed topological space. Let $(C_*(X), \partial, \Delta)$ be the dg coassociative coalgebra of singular chains on $X$  with $\mathbb{Z}$-coefficents equipped with the Alexander-Whitney coproduct $\Delta: C_*(X) \to C_*(X) \otimes C_*(X)$. Let $C'_*(X,b)$ be the subcomplex of $C_*(X)$ generated by those continuous maps $\sigma: |\Delta^n| \to X$ that send the vertices of the topological $n$-simplex $|\Delta^n|$ to $b$. Let $C_*(X,b)$ be the quotient of $C'_*(X,b)$ by the subcomplex of degenerate chains. The chain complex $C_*(X,b)$ inherits a dg coassociative coalgebra structure $(C_*(X,b), \partial, \Delta)$. The main theorem of this article, which we prove in section 3, is the following.

\begin{theorem} If $(X,b)$ is a path-connected pointed space then the cobar construction of $(C_*(X,b), \partial, \Delta)$ is naturally quasi-isomorphic as a dg associative algebra to $C_*(\Omega_bX)$, the singular chains on the topological monoid of (Moore) loops in $X$ based at $b$. 
\end{theorem}

Instead of comparing spectral sequences as in Adams' original paper, which required a connectivity hypothesis, we prove the above theorem by relating the cobar construction to (the left adjoint of) a particular model for the classifying space of a fibrant simplicial groupoid called the \textit{homotopy coherent nerve functor} introduced by Cordier in \cite{Co82} and studied in more detail in \cite{Hi07}, \cite{DuSp11}, \cite{DuSp211}, and \cite{Lu09}.
\medskip

There are other extensions of Adams' cobar theorem in the literature. Kontsevich proposed to formally add strict inverses to all the $1$-simplices of a simplicial complex in the cobar construction of its chains to obtain a chain model for the based loop space \cite{Ko09}. An explicit relationship between this model and the Kan loop group construction was established in \cite{HeTo10}. More recently, in \cite{CHL18}, both \cite{RiZe16} and \cite{HeTo10} are placed in the same framework of derived localizations. 

\section{Preliminaries}
\subsection{Cobar construction}
We assume all (co)algebras are (co)unital. Let $\mathbb{K}$ be a commutative ring with unit and let $\otimes = \otimes_{\mathbb{K}}$. A dg coassociative $\mathbb{K}$-coalgebra $(C_*, \partial: C_* \to C_{*-1}, \Delta: C_* \to C_* \otimes C_*)$ is connected if $C_i=0$ for $i<0$ and the counit defines an isomorphism $C_0 \cong \mathbb{K}$. Let $\mathbf{dgCoalg}_{\mathbb{K}}^0$ be the category of connected dg coassociative $\mathbb{K}$-coalgebras and let $\mathbf{dgAlg}_{\mathbb{K}}$ be the category of dg associative $\mathbb{K}$-algebras.
The \textit{cobar construction} is a functor
$$\mathbf{\Omega}: \mathbf{dgCoalg}_{\mathbb{K}}^0 \to \mathbf{dgAlg}_{\mathbb{K}}$$ defined as follows.
For any $(C_*, \partial, \Delta) \in \mathbf{dgCoalg}^0_{\mathbb{K}}$ define the underlying graded $\mathbb{K}$-algebra of $\mathbf{\Omega}(C_*, \partial, \Delta)$ to be the graded tensor algebra
$$ T(s^{-1}C_{*>0} )= \mathbb{K} \oplus s^{-1}C_{*>0} \oplus (s^{-1}C_{*>0})^{\otimes 2} \oplus (s^{-1}C_{*>0})^{\otimes 3} \oplus ...  ,$$
where $C_{*>0}$ is the quotient of $\mathbb{K}$-modules $C_*/C_0$ and $s^{-1}$ is the shift functor, defined by $(s^{j}M)_i := M_{i-j}$ for any integer $j$ and any graded $\mathbb{K}$-module $M.$  The generators of $T(s^{-1}C_{*>0} )$ are monomials which will be denoted by $[ s^{-1} \sigma_1| s^{-1} \sigma_2|...|s^{-1} \sigma_k],$ where $\sigma_i \in C_{*>0}$. Multiplication is given by concatenation of monomials. The differential of $\mathbf{\Omega}(C_*, \partial, \Delta)$ is defined by extending the linear map
$$- s^{-1} \circ \partial  \circ s^{+1} + (s^{-1} \otimes s^{-1}) \circ \Delta \circ s^{+1}: s^{-1}C_{*>0} \to s^{-1}C_{*>0} \oplus (s^{-1}C_{*>0} \otimes s^{-1}C_{*>0})$$
as a derivation (by the Leibniz rule) to all monomials to obtain a linear map of degree $-1$
$$D: T(s^{-1}C_{*>0} )\to T(s^{-1}C_{*>0} ).$$ The compatibility of $\partial$ and $\Delta$, the coassociativity of $\Delta$, and $\partial^2=0$, together imply $D^2=0$.

\begin{remark} The cobar functor does not send quasi-isomorphisms of dg coalgebras to quasi-isomorphisms of dg algebras. A counterexample is discussed in Proposition 2.4.3 of \cite{LoVa12}. 
\end{remark}

\subsection{Rigidification and homotopy coherent nerve}
Let $\mathbf{sSet}$ and $\mathbf{sCat}$ be the categories of simplicial sets and categories enriched over the monoidal category of simplicial sets with cartesian product, respectively. We call the objects of $\mathbf{sCat}$ \textit{simplicial categories.}.  Denote by $
\Delta^n \in \mathbf{sSet}$ the standard $n$-simplex. We recall the definition of the \textit{rigidification functor} 
 $$\mathfrak{C}: \mathbf{sSet} \to \mathbf{sCat}.$$ 
following \cite{Lu09}. Given integers $0 \leq  i \leq j$ denote by $P_{i,j}$ the category whose objects are all the subsets of $\{i, i+1, ..., j\}$ containing both $i$ and $j$ and morphisms are inclusions. %There are natural isomorphisms $P_{i,j} \cong \mathbf{1}^{j-i-1}$ if $i<j$ and $P_{i,i} \cong \mathbf{1}^0$, where $\mathbf{1}^n$ denotes the cartesian product of $n$ copies of the category $\mathbf{1}=\{0 \to 1\}$ with two objects and one non-identity morphism and $\mathbf{1}^0$ is the category with one object and one morphism.%
For each integer $n \geq 0$ define $\mathfrak{C}(\Delta^n) \in \mathbf{sCat}$ to have the set $\{0, ... , n\}$ as objects and if $i \leq j$, define $\mathfrak{C}(\Delta^n)(i,j):= N(P_{i,j})$, where $N$ is the ordinary nerve functor.\footnote{Note there is a natural isomorphism of simplicial sets $\mathfrak{C}(\Delta^n)(i,j) \cong (\Delta^1)^{\times (j-i-1)}$ if $i<j$ and  $\mathfrak{C}(\Delta^n)(i,i) \cong \Delta^0$.}  If $j < i$, $\mathfrak{C}(\Delta^n)(i,j):= \emptyset.$ The composition law in $\mathfrak{C}(\Delta^n)$ is induced by the functor $P_{j,k} \times P_{i,j} \to P_{i,k}$ given by union of sets. The assignment $[n] \mapsto \mathfrak{C}(\Delta^n)$ defines a cosimplicial object in $\mathbf{sCat}$. For any $S \in \mathbf{sSet}$ define 
 $$\mathfrak{C}(S):= \underset{{\Delta^n \to S} }{\text{colim }} \mathfrak{C}(\Delta^n).$$
The functor $\mathfrak{C}$ is the left adjoint of the \textit{homotopy coherent nerve functor}
$$\mathfrak{N}: \mathbf{sCat} \to \mathbf{sSet},$$ which is given by
$$\mathfrak{N}(\mathcal{C})_n=\text{Hom}_{\mathbf{sCat}}(\mathfrak{C}(\Delta^n), \mathcal{C}).$$
We recall a description of the mapping spaces $\mathfrak{C}(S)(x,y)$ given in \cite{DuSp11}. By a \textit{necklace} we mean a simplicial set of the form $T=\Delta^{n_1} \vee ... \vee \Delta^{n_k}$ obtained from an ordered sequence of standard simplices, which we call the beads of $T$, with $n_i >0$ by gluing the final vertex of one to the first vertex of its successor. Define the dimension of a necklace by $\text{dim}(T):= n_1+...+n_k-k$. Denote by $\alpha_T$ and $\omega_T$ the first and last vertices of $T$ using the natural ordering on the set vertices. We declare $\Delta^0$ to be a necklace of dimension $0$. Necklaces form a category $\mathbf{Nec}$ with morphisms being maps of simplicial sets $f: T \to T'$ such that $f(\alpha_T)=\alpha_{T'}$ and $f(\omega_T)=\omega_{T'}$. For $S \in \mathbf{sSet}$ and $x,y\in S_0$ denote by $(\mathbf{Nec} \downarrow S)_{x,y}$ the full subcategory of the over category $\mathbf{Nec} \downarrow S$ consisting of those maps $f: T\to S$ such that $f(\alpha_T)=x$ and $f(\omega_T)=y$. The category $\mathbf{Nec}$ has a (non-symmetric) monoidal structure $$\vee: \mathbf{Nec} \times \mathbf{Nec} \to \mathbf{Nec}$$ given by concatenation of necklaces with identity object $\Delta^0$. 

\begin{proposition} (Proposition 4.3 \cite{DuSp11}) For any $S \in \mathbf{sSet}$ and $x,y \in S_0$, there are natural isomorphisms of simplicial sets
$$\mathfrak{C}(S)(x,y) \cong \underset{(f: T \to S) \in (\mathbf{Nec} \downarrow S)_{x,y}}{\emph{colim}} \mathfrak{C}(T)(\alpha_T, \omega_T) \cong \underset{(f: T \to S) \in (\mathbf{Nec} \downarrow S)_{x,y}}{\emph{colim}}  (\Delta^{1})^{\times\emph{dim}(T)}$$ The composition map $\mathfrak{C}(S)(y,z) \times \mathfrak{C}(S)(x,y) \to \mathfrak{C}(S)(x,z)$ is given by $[(f': T' \to S), \sigma'] \times [(f: T\to S), \sigma] \mapsto [(f\vee f': T\vee T'\to S), \sigma \times \sigma']$, where the notation $[(f:T\to S),\sigma] \in \mathfrak{C}(S)(x,y)$ denotes the equivalence class of the pair $((f: T\to S), \sigma)$ in the colimit. 
\end{proposition}
\begin{remark} The second isomorphism in the above proposition, as explained in Corollary 3.8 of \cite{DuSp11}, is induced by a natural isomorphism $$\phi_T: \mathfrak{C}(T)(\alpha_T, \omega_T)  \cong (\Delta^1)^{\times \text{dim}(T)}.$$
An essential feature of $\phi$ is that if $\iota: T' \hookrightarrow T$ is an injection in $\mathbf{Nec}$ and $\text{dim}(T')= \text{dim}(T)-1$ then the induced map of simplicial cubes $$\phi_T \circ \mathfrak{C}(\iota) \circ \phi_{T'}^{-1}: (\Delta^1)^{ \times \text{dim}(T)-1} \to (\Delta^1)^{\times \text{dim}(T)}$$
is a cubical face inclusion, namely,  $(\Delta^1)^{\times \text{dim}(T)-1}$ is mapped injectively into one of the boundary faces of $(\Delta^1)^{\times \text{dim}(T)}$. Moreover, all cubical face inclusions are realized by maps of necklaces in this way. 

\end{remark}
\begin{remark}
The adjunction given by the functors $\mathfrak{C}: \mathbf{sSet} \to \mathbf{sCat}$ and $\mathfrak{N}: \mathbf{sCat} \to \mathbf{sSet}$ defines a Quillen equivalence between the Joyal model structure on $\mathbf{sSet}$ and the Bergner model structure on $\mathbf{sCat}$. This is proven in Chapter 2 of \cite{Lu09} and a different proof is given in \cite{DuSp211}. The Quillen model structure on $\mathbf{sSet}$ is a Bousfield localization of the Joyal model structure. In particular, a weak homotopy equivalence between simplicial sets which are fibrant objects in the Quillen model structure (Kan complexes) is a weak equivalence in the Joyal model structure. This is proven directly in Proposition 17.2.8 of \cite{Ri14}. 
\end{remark}

We recall Corollary 2.6.3 of \cite{Hi07} where the relationship between homotopy coherent nerve $\mathfrak{N}$ and the classifying space functor was established. A \textit{fibrant groupoid} $\mathcal{C}$ is a simplicial category such that for all $x,y \in \mathcal{C}$, $\mathcal{C}(x,y)$ is a Kan complex and each $\sigma\in \mathcal{C}(x,y)_0$ is invertible up to homotopy. For any $\mathcal{C} \in \mathbf{sCat}$, let $\mathcal{B}(\mathcal{C})$ be the simplicial set determined by the diagonal of the bisimplicial set obtained by applying (level-wise) the ordinary nerve to $\mathcal{C}$.
\begin{proposition} \cite{Hi07}
There is a natural map $\mathcal{B}(\mathcal{C}) \to \mathfrak{N}(\mathcal{C})$  which is a weak homotopy equivalence when $\mathcal{C}$ is a fibrant groupoid. 
\end{proposition}
The functor $\mathcal{B}$ is a model for the classifying space functor of a simplicial monoid, as discussed in \cite{Se68} and p. 86 of \cite{Qu73}. 

\section{Proof of the main theorem}
Let $\mathbf{sSet}^0$ be the full subcategory of $\mathbf{sSet}$ consisting of simplicial sets with a single vertex and let $\textbf{sMon}$ be the category of simplicial monoids. For any $K \in \mathbf{sSet}^0$, the simplicial category $\mathfrak{C}(K)$ has a single object so, for simplicity, we will consider $\mathfrak{C}(K)$ as an object in $\mathbf{sMon}$. Denote by $| \cdot |: \mathbf{sSet} \to \mathbf{Top}$ the geometric realization functor and by $\text{Sing}: \mathbf{Top} \to \mathbf{sSet}$ the singular complex functor. In this section we deduce Theorem 1 as a consequence of several results. 

\begin{proposition} If $K \in \mathbf{sSet}^0$ is a Kan complex with $K_0=\{x\}$, the simplicial monoids $\mathfrak{C}(K)$ and $\emph{Sing}(\Omega_x|K|)$ are naturally weakly homotopy equivalent. 
\end{proposition} 
\begin{proof}
Since $\text{Sing}(\Omega_x|K|)$ is a fibrant groupoid, by Proposition 5 there is a natural weak homotopy equivalence
$$\mathcal{B}( \text{Sing}(\Omega_x|K|)) \xrightarrow{\simeq}  \mathfrak{N}( \text{Sing}(\Omega_x|K|) ).$$
We know there is a natural weak homotopy equivalence of Kan complexes 
$$K \xleftarrow{\simeq} \mathcal{B}( \text{Sing}(\Omega_x|K|) ),$$ which follows from the well known natural weak homotopy equivalence relating the classifying space functor and the based loop space functor, see Lemma 15.4 of \cite{Ma75} for an explicit formula. Note that $\mathfrak{N}( \text{Sing}(\Omega_x|K|) )$ is a Kan complex since it is a quasi-category whose homotopy category is a groupoid. By Remark 5, the functor $\mathfrak{C}: \mathbf{sSet} \to \mathbf{sCat}$ sends weak homotopy equivalences between Kan complexes to weak equivalences of simplicial categories. Hence by applying $\mathfrak{C}$  we obtain natural weak equivalences of simplicial categories
$$\mathfrak{C}(K) \xleftarrow{\simeq} \mathfrak{C}(\mathcal{B}( \text{Sing}(\Omega_x|K|))) \xrightarrow{\simeq} \mathfrak{C}( \mathfrak{N}( \text{Sing}(\Omega_x|K|))).$$ Finally, since $\mathfrak{C}$ and $\mathfrak{N}$ define a Quillen equivalence between model categories and $\text{Sing}(\Omega_x|K|)$ is a fibrant simplicial category, 
it follows that the counit of the adjunction induces a weak equivalence of simplicial categories
$$ \mathfrak{C} (\mathfrak{N}( \text{Sing}(\Omega_x|K|) ) )\xrightarrow{\simeq} \text{Sing}(\Omega_x|K|).$$
Thus $ \mathfrak{C}(K)$ and $ \text{Sing}(\Omega_x|K|)$ are naturally weakly equivalent as simplicial categories.
\end{proof}

% Alternate proof: By Corollary 6, there exists a natural weak homotopy equivalence  of Kan complexes $$\xi: \mathfrak{N}( \text{Sing}(\Omega_x |K|)) \to K.$$  Since $\mathfrak{C}$ sends weak homotopy equivalences between Kan complexes to weak equivalences of simplicial categories (as shown in Proposition 17.2.8 of \cite{Rie14}) the induced map $$\mathfrak{C}(\xi):\mathfrak{C}(\mathfrak{N}( \text{Sing}(\Omega_x |K|))) \to \mathfrak{C}(K)$$ is a weak equivalence of simplicial categories. By Proposition 5.8 in \cite{DuSp211} or Theorem 2.2.0.1 of \cite{Lu09}, there is a natural weak equivalence of simplicial categories $$\varepsilon: \mathfrak{C} ( \mathfrak{N}( \text{Sing}(\Omega_x |K|))) \to \text{Sing}(\Omega_x |K|).$$ It follows that $\mathfrak{C}(K)$ and $\text{Sing}(\Omega_x |K|)$ are naturally weakly equivalent as simplicial monoids.

Let $\mathbf{Ch}_{\mathbb{Z}}$ be the category of chain complexes of abelian groups and $C_*: \mathbf{sSet} \to \mathbf{Ch}_{\mathbb{Z}}$ the normalized chains functor. The Eilenberg-Zilber natural transformation $C_*(S) \otimes C_*(S')  \to  C_*(S \times S')$ makes $C_*: \mathbf{sSet} \to \mathbf{Ch}_{\mathbb{Z}}$ into a lax monoidal functor. Thus, there is an induced functor $C_*: \mathbf{sMon} \to \mathbf{dgAlg}_{\mathbb{Z}}$. More precisely, if $M \in \mathbf{sMon}$ we have natural maps
$$C_*(M) \otimes C_*(M) \xrightarrow{EZ} C_*(M \times M) \to C_*(M),$$
where the first map is induced by the Eilenberg-Zilber shuffle map and the second one by the monoid structure of $M$, making $C_*(M)$ into a dg associative algebra.

Since the normalized chains functor sends weak homotopy equivalences to quasi-isomorphisms, Proposition 6 implies the following

\begin{corollary}  If $K \in \mathbf{sSet}^0$ is a Kan complex, the dg associative algebras $C_*( \mathfrak{C}(K))$ and $C_*^{\emph{sing}}(\Omega_x|K|)$ are naturally quasi-isomorphic. 
\end{corollary}
We now describe $C_*(\mathfrak{C}(K))$ in more detail. Since the normalized chains functor commutes with colimits, by Proposition 3 there is a natural isomorphism of chain complexes 
$$C_*(\mathfrak{C}(K)) \cong \underset{(f: T \to K) \in \mathbf{Nec} \downarrow K}{\text{colim}}  C_*((\Delta^{1})^{\times\text{dim}(T)}).$$  Therefore, any element of $C_n(\mathfrak{C}(K))$ is given by the equivalence class $[f,\sigma]$ of a pair $(f: T \to K,\sigma)$ where $(f: T \to K) \in \mathbf{Nec} \downarrow K$ and $\sigma \in C_n( (\Delta^{1})^{\times \text{dim}(T)})$. 
%Furthermore, the equivalence class of any generator in $C_n(\mathfrak{C}(K))$, other than $u:=[s_0(x),*] \in C_0(\mathfrak{C}(K))$ (where $s_0(x): \Delta^1 \to K$ is the degenerate $1$-simplex at $x$ and $* \in C_0(\Delta^0)$ the unique non-degenerate generator) may be represented by $[f,\sigma]$ where
%\begin{itemize}
%\item $f : \Delta^{n_1} \vee... \vee \Delta^{n_k} \to K$ has the property that the image under $f$ of each bead is a non-degenerate simplex of $K$ and
%\item $\sigma \in C_n((\Delta^1)^{\times n_1 + ... n_k-k})$ is the generator corresponding to some non-degenerate $n$-simplex of $(\Delta^1)^{\times n_1 + ... n_k-k}$
%\end{itemize}
The differential $\partial: C_*(\mathfrak{C}(K)) \to C_{*-1}(\mathfrak{C}(K))$ is given by $\partial[f,\sigma]=[f,\partial \sigma]$. 
The algebra structure $$C_*(\mathfrak{C}(K)) \otimes C_*(\mathfrak{C}(K)) \to C_*(\mathfrak{C}(K))$$ may be described on any two classes $[g,\sigma]$ and $[g',\sigma']$ by
$$[g,\sigma] \otimes [g', \sigma'] \mapsto [g' \vee g, EZ(\sigma' , \sigma) ],$$
where $\vee$ denotes the concatenation (or wedge) product of necklaces and $EZ: C_*(S ) \otimes (S') \to C_*(S \times S')$ the Eilenberg-Zilber map. The class $u:=[s_0(x),*] \in C_0(\mathfrak{C}(K))$ is the unit for this product, where $s_0(x): \Delta^1 \to K$ is the degenerate $1$-simplex at $x$ and $* \in C_0(\Delta^0)$ the unique generator. 
\\

The Alexander-Whitney coproduct $\Delta: C_*(K) \to C_*(K) \otimes C_*(K)$ is defined by
$$ \Delta(\sigma) = \sum_{i=0}^n \sigma|_{[0,...,i]} \otimes \sigma|_{[i,...,n]}$$
for any $\sigma \in C_n(K)$, where $\sigma|_{[0,...,i]} \in C_i(S)$ and $\sigma|_{[i,...,n]} \in C_{n-i}(S)$ denote the simplices obtained by restricting $\sigma$ to its first $i$-dimensional face and its last $(n-i)$-dimensional face, respectively. If $K \in \mathbf{sSet}^0$, then $(C_*(K),\partial, \Delta)$ is a connected dg coassociative coalgebra. 

\begin{theorem} For any $K \in \mathbf{sSet}^0$, the dg associative algebras $\mathbf{\Omega}(C_*(K), \partial, \Delta)$ and $C_*( \mathfrak{C}(K))$ are naturally quasi-isomorphic.
\end{theorem}
\begin{proof} 
Define a functor $F: \mathbf{Nec} \to \mathbf{Ch}_{\mathbb{Z}}$ as follows. On objects, $F$ is given by setting $F(\Delta^0)= C^{\square}_*(\square^0)$ and, on all other neckalces $T \in \mathbf{Nec}$, $$F(T)= C^{\square}_*( \square^{\text{dim}(T)} ),$$ where $C^{\square}_*( \square^{ n} )$ denotes the chain complex of normalized cubical chains on the standard $n$-cube $\square^n$ considered as a cubical set, or equivalently, the cellular chains on the topological $n$-cube with its standard (cubical) cell structure. We now define $F$ on morphisms.   
The morphisms in $\mathbf{Nec}$ are generated via the monoidal structure $\vee: \mathbf{Nec} \times \mathbf{Nec} \to \mathbf{Nec}$ by the following types of necklace maps. 

\begin{enumerate}
\item $\partial_j \colon \Delta^{n} \hookrightarrow \Delta^{n+1}$ for $j = 1, \dots, n$,
\item $\Delta_{[j], [n+1-j]} \colon \Delta^{j} \vee \Delta^{n+1-j} \hookrightarrow \Delta^{n+1}$ for $j = 1, \dots, n$,
\item $s_j \colon \Delta^{n+1} \twoheadrightarrow \Delta^{n}$ for $j = 0, \dots, n$ and $n>0$, and
\item $s_0 \colon \Delta^1 \twoheadrightarrow \Delta^0$.
\end{enumerate}
Maps of type $(1)$ are the codimension $1$ injective maps between necklaces of length $1$ given by simplicial co-face maps between standard simplices. Maps of type $(2)$ are the codimension $1$ injective maps from a necklace of length $2$ to a necklace of length $1$ given by the first $j$-th coface and last $(n+1-j)$-th coface maps of standard simplices.  Maps $(3)$ and $(4)$ are the surjective maps of necklaces given by the corresponding simplicial co-degeneracy maps between standard simplicies. 

We define $$F(\partial_j )= C^{\square}_*(\delta^0_s): C^{\square}_*(\square^{n-1}) \to C^{\square}_*(\square^{n}), \text{   } 1 \leq s \leq n$$
and
$$F(\Delta_{[j], [n+1-j]})= C^{\square}_*(\delta^1_s): C^{\square}_*(\square^{n-1}) \to C^{\square}_*(\square^{n}), \text{   } 1 \leq s \leq n,$$
where $\delta^0_s: \square^{n-1} \hookrightarrow \square^n$ and $\delta^1_s: \square^{n-1} \hookrightarrow \square^n$, for $1 \leq s \leq n$ are the cubical co-face inclusion maps given by inserting $0$ or $1$ in the $s$ coordinate, respectively. Finally, define $F$ to be the trivial map on morphisms of type $(3)$ and the identity map on morphisms of type $(4)$. These definitions completely determine a monoidal functor $F: \mathbf{Nec} \to \mathbf{Ch}_{\mathbb{Z}}$.

Now consider the functor $F^{\square}_*: \mathbf{sSet}^0 \to \mathbf{dgAlg}_{\mathbb{Z}}$ induced by $F$ as follows.  The underlying chain complex is
$$F^{\square}_*(K) := \underset{(f: T \to K) \in \mathbf{Nec} \downarrow K}{\text{colim}} F(T) = \underset{(f: T \to K) \in \mathbf{Nec} \downarrow K}{\text{colim}}  C^{\square}_*( \square^{\text{dim}(T)} ). $$

Any generator of $F^{\square}_n(K)$, other than $u=[s_0(x),\iota_0] \in F_0^{\square}(K)$, is an equivalence class which may be represented \textit{uniquely} as $[f, \iota_n]$, where:
\begin{itemize}
\item $(f: T \to K)  \in \mathbf{Nec} \downarrow K$ has the property that the image of every bead of $T$ under $f$ is a non-degenerate simplex of $K$, $\text{dim}(T)=n$, and
\item $\iota_n \in C^{\square}_n(\square^n)$ is the generator corresponding to the unique top dimensional non-degenerate $n$-cube of $\square^n.$
\end{itemize}
Any object $(f: \Delta^{n_1} \vee ... \vee \Delta^{n_k} \to K) \in \mathbf{Nec} \downarrow K$  is equivalent to an ordered monomial $(f_1 \vee... \vee f_k)$ where $f_i \in K_{n_i}$ is the $n_i$-simplex determined by the restriction of $f$ to the $i$-th bead $\Delta^{n_i}$. Using this identification, any class in $F^{\square}_n(K)$ may be written uniquely as $[(f_1 \vee... \vee f_k), \iota_n]$, where each $f_i \in K_{n_i}$ is non-degenerate and $n_1+...+n_k-k=n$. 

The differential $\partial^{\square}: F^{\square}_*(K) \to F^{\square}_{*-1}(K)$, which is given by $\partial^{\square}[ (f_1 \vee... \vee f_k),\iota_n]= [(f_1 \vee... \vee f_k), \partial^{\square}\iota_n],$ may now be described as: 
\begin{eqnarray*}
\partial^{\square}[(f_1 \vee... \vee f_k),\iota_n]= \sum_{i=1}^{k} \sum_{j=1}^{n_i-1} \pm [(f_1 \vee... \vee f_i|_{[0,...,j]}\vee f_i|_{[j,...,n_i]} \vee ... \vee f_k), \iota_{n-1}]
\\
- \sum_{i=1}^k \sum_{l=1}^{n_i-1}\pm [ ( f_1 \vee .... \vee f_i|_{[0,...,\hat{l},...,n_i]}\vee ... \vee f_k), \iota_{n-1}].
\end{eqnarray*}
The above formula follows from the cubical boundary formula $$\partial^{\square}\iota_n= \sum_{s=1}^n(-1)^{s}(C^{\square}_*(\delta^1_s)(\iota_{n-1})- C^{\square}_*(\delta^0_s)(\iota_{n-1}))$$ together with the fact that, given $(f_1 \vee .... \vee f_k):T \to K$, the cubical co-face maps $C_*^{\square}(\delta^0_s)$ and $C_*^{\square} (\delta^1_s)$ are realized uniquely through the functor $F$ by inclusion maps in $\mathbf{Nec} \downarrow K$ 
$$(f_1 \vee... \vee f_i|_{[0,...,j]}\vee f_i|_{[j,...,n_i]} \vee ... \vee f_k) \hookrightarrow (f_1 \vee...\vee f_i \vee... \vee f_k)$$
and
$$(f_1 \vee .... \vee f_i|_{[0,...,\hat{l},...,n_i]}\vee ... \vee f_k) \hookrightarrow (f_1 \vee...\vee f_i \vee... \vee f_k).$$
 
 The graded associative algebra structure on $F^{\square}_*(K)$ is 
 $$ [(f_1 \vee... \vee f_k),\iota_n] \otimes [(g_1 \vee... \vee g_l),\iota_m] \mapsto [ (f_1 \vee... \vee f_k \vee g_1 \vee... \vee g_l),\iota_{n+m}],$$ with unit  $u=[s_0(x), \iota_0]$. It is clear that the construction of $F^{\square}_*(K)$ is functorial with respect to maps in $\mathbf{sSet}^0$. The theorem will follow from the following claims:
 \\
 \\
 \textit{Claim 1}: There is a natural quasi-isomorphism of dg algebras $$\psi: F^{\square}_*(K) \xrightarrow{\simeq} C_*(\mathfrak{C}(K)).$$
 \\
 \textit{Claim 2}: There is a natural isomorphism of dg algebras $$\varphi: F^{\square}_*(K) \xrightarrow{\cong} \mathbf{\Omega}(C_*(K), \partial, \Delta).$$
 \\
 \\
\textit{Proof of Claim 1:} Define $\psi$ on any generator $[ (f_1 \vee .... \vee f_k), \iota_n ] \in F^{\square}_n(K)$ where the $f_i \in K_{n_i}$ are non-degenerate simplices and $n_1+...+n_k-k=n$ by
$$\psi  [ (f_1 \vee .... \vee f_k), \iota_n ] := [ (f_1 \vee .... \vee f_k), e^{\times n} ],$$
where $e^{\times n} \in C_*((\Delta^1)^{\times n})$ denotes $e \times ... \times e$ ($n$-times) for $e \in C_1(\Delta^1)$ the unique non-degenerate generator and $\times$ is the Eilenberg-Zilber map. It is easily checked that $\psi$ is a map of differential graded algebras. The fact that $\psi$ is a quasi-isomorphism follows from a standard application of the Acyclic Models Theorem in the following set up. Let $[\mathbf{Nec}^{op}, \mathbf{Set}]$ be the category whose objects are functors $\mathbf{Nec}^{op} \to \mathbf{Set}$ and morphisms are natural transformations. Let $\mathcal{Y}: \mathbf{sSet}_{*,*} \to [\mathbf{Nec}^{op}, \mathbf{Set}]$ be defined on any double pointed simplicial set $(S,x,y) \in \mathbf{sSet}_{*,*}$ by $$\mathcal{Y}(S,x,y)(T):= \mathbf{sSet}_{*,*}((T, \alpha_T, \omega_T), (S,x,y)).$$ The underlying chain complex of $F^{\square}_*(K)$ is precisely $\mathcal{F}(\mathcal{Y}(K,x,x))$, where

$$\mathcal{F}: [\mathbf{Nec}^{op}, \mathbf{Set}] \to \mathbf{Ch}_{\mathbb{Z}}$$ is the functor
defined by
$$\mathcal{F}(X):=  \underset{ (\mathcal{Y}(T,\alpha_T,\omega_T) \to X) \in \mathbf{Nec} \downarrow X  }{\text{colim}} C^{\square}_*(\square^{\text{dim}(T)}).$$
Similarly, the underlying chain complex of $C_*(\mathfrak{C}(K))$ is precisely $\mathcal{G}(\mathcal{Y}(K,x,x))$, where
$$\mathcal{G}: [\mathbf{Nec}^{op}, \mathbf{Set}] \to \mathbf{Ch}_{\mathbb{Z}}$$ is the functor
defined by
$$\mathcal{G}(X):=  \underset{ (\mathcal{Y}(T,\alpha_T,\omega_T) \to X) \in \mathbf{Nec} \downarrow X  }{\text{colim}} C_*((\Delta^1)^{\times \text{dim}(T)}).$$
We may now use $\{ \mathcal{Y}(T, \alpha_T, \omega_T) \}_{T \in \mathbf{Nec}}$ as a collection of models in $[\mathbf{Nec}^{op}, \mathbf{Set}]$, which are acyclic for both $\mathcal{F}$ and $\mathcal{G}$, to conclude that $\psi: \mathcal{F}(\mathcal{Y}(K,x,x))\to \mathcal{G}(\mathcal{Y}(K,x,x))$ is a quasi-isomorphism. \hfill $\Box$
\\

\textit{Proof of Claim 2:} Let
$$\varphi [s_0(x), \iota_0] := 1.$$
If $f_1 \in K_1$ is non-degenerate, let
$$\varphi [ f_1, \iota_0] :=  [s^{-1}\bar{f_1}] - 1$$ where $\bar{f_1} \in C_1(K)$ denotes the class of $f_1 \in K_1$ in the normalized chain complex of $K$. 
For any $f_1 \in K_n$ with $n>1$, let
$$\varphi [f_1, \iota_{n-1}] := [s^{-1} \bar{f_1}].$$
Extend the above as an algebra map to define $\varphi$ on any generator $[ (f_1 \vee ... \vee f_k), \iota_n] \in F^{\square}_n(K)$. It follows from the formula for $\partial^{\square}[ (f_1 \vee ... \vee f_k), \iota_n]$ that $\varphi$ is compatible with differentials. It is easily checked that $\varphi$ is an isomorphism. \hfill $\Box$
\\
\\
\textit{Proof of Theorem 1.} The main theorem now follows by applying Corollary 7 and Theorem 8 to $K=\text{Sing}(X,b)$, the sub Kan complex of $\text{Sing}(X)$ consisting of continuous maps $\sigma: |\Delta^n| \to X$ which send the vertices of $|\Delta^n|$ to $b \in X$. 
\end{proof}

\bibliographystyle{plain}

\end{document}